\documentclass[11pt,twoside, final]{amsart}
\copyrightinfo{0}{Iranian Mathematical Society}
\pagespan{1}{\pageref*{LastPage}}
\usepackage{etoolbox,lastpage}
\commby{}
\usepackage{amsmath,amsthm,amscd,amsfonts,amssymb,enumerate}
\usepackage{graphicx}		
\usepackage{color}
\usepackage[colorlinks]{hyperref}

\def\opn#1#2{\def#1{\operatorname{#2}}}
\opn\chara{char} \opn\length{\ell} \opn\pd{pd} \opn\rk{rk}
\opn\projdim{proj\,dim} \opn\injdim{inj\,dim} \opn\rank{rank}
\opn\depth{depth} \opn\grade{grade} \opn\height{height} \opn\supp{supp}
\opn\embdim{emb\,dim} \opn\codim{codim}

\opn\Tr{Tr} \opn\bigrank{big\,rank}
\opn\superheight{superheight}\opn\lcm{lcm}
\opn\trdeg{tr\,deg}%\emph{
\opn\reg{reg} \opn\lreg{lreg} \opn\ini{in} \opn\indeg{indeg}
\opn\Syz{Syz} \opn\size{size} \opn\reg{reg}
 \opn\link{link}\opn\pd{pd}
\def\m{{\mathfrak{m}} }

\newtheorem{theorem}{Theorem}[section]

\newtheorem{corollary}[theorem]{Corollary}
\theoremstyle{definition}
\newtheorem{definition}[theorem]{Definition}
\newtheorem{example}[theorem]{Example}
\theoremstyle{remark}
\newtheorem{remark}[theorem]{Remark}
\numberwithin{equation}{section}
\begin{document}
%% The title of the paper goes here.  Edit your title.

\title[Minimal free resolution  by  mapping cone]{Minimal free resolution of monomial ideals by iterated mapping cone}
%% Now edit the following to give First Author name and address:
%% $^*$ for the corresponding author.

\author[L. Sharifan]{Leila Sharifan$^*$}
\address[Leila Sharifan]{Faculty of Mathematics and Computer Sciences‎, ‎Hakim Sabzevari University‎, ‎P.O‎. ‎Box 397‎, ‎Sabzevar‎, ‎Iran, and
 School of Mathematics, Institute for research in Fundamental Sciences (IPM), P. O. Box: 19395-5746, Tehran, Iran.}
\email{leila-sharifan@aut.ac.ir}

%\author[S. Author]{Second Author $^*$}
%\address[Second Author]{Complete Address}
%\email{email@some.ir}
%% If there are three of more authors they are added in the obvious way.

  \thanks{$^*$Corresponding author}
%------------------------------------------------------------------------------------%
%%
%% Use the following command to make the title for the paper.
%
 %\CoverPage

 \maketitle
%
%%% The following environment is needed for the abstract.
%%%

\begin{abstract}
In this paper we study minimal free resolutions of some classes of
monomial ideals. we first give a
sufficient condition to check the minimality of the resolution
obtained by the  mapping cone. Using it, we  obtain the Betti numbers of  max-path ideals of
rooted trees and ideals containing
powers of variables. In particular, we discuss  about resolutions of  ideals of the form $J_{\mathcal{H}}+(x_{i_1}^2,\ldots, x_{i_m}^2)$ where $J_{\mathcal{H}}$ is the edge ideal of a  hypergraph $\mathcal{H}$.\\
\textbf{Keywords:}   Mapping cone,
regularity, max-path ideal,
 edge ideal of hypergraph, independent number. \\
\textbf{MSC(2010):}  Primary: 13D02; Secondary: 05E40,05C65.
\end{abstract}

\section{\bf Introduction}

 Let ${\bf{k}}$ be a field, $R={\bf{k}}[x_1,\ldots, x_n]$ the polynomial ring
 in $n$ variables, and $I$ a graded  ideal.
Finding algebraic properties  of $I$ like regularity, projective dimension and  depth is a central problem in commutative algebra and algebraic geometry. Computing the (graded) minimal free resolution of $I$ is the key to find these invariants. However, describing the precise minimal free resolution of an ideal, even in the case that $I$ is a square-free monomial ideal is not an easy problem and when $I$ is  not a square monomial ideal the problem is more difficult. An standard tool to compute a free resolution of  an ideal is {\it{iterated mapping cone}}. In the monomial case, several well known resolution arise as iterated mapping cone. For example, the Taylor resolution \cite{Ta}, the Eliahou-Kervaire resolution of stable monomial ideals \cite{EC} and resolution of monomial ideals with linear quotients \cite{HT}.

In this paper,   by iterated mapping cone, we study minimal free resolution of
some class of monomial ideals.   Note that, in general the result of the mapping cone is not a minimal free resolution. The importance of our work is that we find a sufficient condition for {\it minimality} of the resolution obtained by this tool. Then we focus to the monomial case and study the particular classes   {\it max-path
ideals of rooted trees} and monomial ideals containing some {\it powers of
variables.}

The paper proceeds as follows. After reviewing some algebraic tools
in  Section 2, in Theorem \ref{criteria1} we  show that for a
graded ideal $I$ and a homogeneous polynomial $f$   which does not belong to $I$,
the minimal free resolution of $R/I+(f)$ is obtained by the mapping
cone  provided that we can decompose $f$ as $f=h_1h_2$ where $h_i$s are  homogeneous polynomials, $\deg(h_2)>0$ and $(I:f)=(I:h_1)$.  Theorem \ref{criteria1}  leads us to introduce  the class of monomial ideals
 of  decreasing type. We say
$I=(u_1,\ldots,u_m)$ is {\it of decreasing type} with respect to the
order $u_1,\ldots, u_m$ of its generators, if for each $u_j$ there exists $x_i\in
\supp(u_j)$ such that $\deg_{x_i}(u_j)>\deg_{x_i}(u_r)$ for all
$r<j$. In
this situation  the minimal free resolution of $R/I$ is obtained by
iterated mapping cone (see Corollary \ref{dec type}).

In  the next sections we apply Theorem \ref{criteria1} and
Corollary  \ref{dec type} to study
homological properties  of  max-path
ideals of rooted trees and
monomial ideals containing powers of some variables. Beside this
goal we present some other interesting properties of the mentioned
classes of ideals.

When $I$ is a square-free  monomial ideal, it is possible to associate to $I$ a combinatorial object such as graph or hypergraph and encode algebraic properties of $I$ in terms of combinatorial properties of corresponding  object.
 It is also natural to start
by a combinatorial object and associate  to it an ideal. The classes of
path ideals of graphs in \cite{CD} and  max-path ideals of
trees in \cite{ShKhN} are defined in this way.

Let $T$ be a rooted tree, the max-path ideal of $T$, denoted
$PI(T)$,
 is defined as $$PI(T)=(x_{i_1}\cdots x_{i_t}\ ;\ {i_1} ,\ldots,
{i_t}{\text{ is a maximal path in}} \ T) \subseteq R,$$ where by a
{\it maximal path} we mean a path between the root of tree and one
of its leaves. In Theorem \ref{rooted tree main1} we give an interesting application of
Corollary  \ref{dec type}. We  show that  $PI(T)$ is of decreasing type and  compute
Betti numbers, regularity, and projective dimension of $R/PI(T)$ in
terms of the number of vertices of $T$ and the number of its leaves.

Next, we consider $PI(T)$ as the
facet ideal of a simplicial complex. denoting  by $\Delta_{PI(T)}$
the simplicial complex corresponding to $PI(T)$, in Theorem
\ref{simplicial tree} we show that $\Delta_{PI(T)}$ is a simplicial
tree. This shows that
 $R/PI(T)$
is sequentially Cohen-Macauly  and so, $PI(T)^\vee$ is a
componentwise linear ideal (Theorem \ref{rooted tree main 2}).

Section 4 is devoted to the study of monomial ideals that
contain some powers of some variables.  Assume that  $I=J+( x_{i_1}^{a_{i_1}},\ldots,x_{i_m}^{a_{i_m}})$ where $J$ is a monomial ideal and $G(I)=G(J)\cup \{x_{i_1}^{a_{i_1}},\ldots,x_{i_m}^{a_{i_m}}\}$. In Theorem
\ref{gereralized peeva} we give a formula for the  graded Betti numbers of $R/I$.  We remark that this result is a straight forward consequence of \cite[Theorem 6.1]{MPS}
(see also \cite[Theorem 2.1]{MPS}). Here we give an easier  proof  for it as an application of
Theorem \ref{criteria1}.

Next we apply Theorem \ref{gereralized peeva} to study monomial
ideals of the form $I=J+(x_{i_1}^2,\ldots ,x_{i_m}^2)$ where $J$ is
a square-free monomial ideal. We consider $J$ as  edge ideal of a hypergraph $\mathcal{H}$. In Therem \ref{hypergraph main theorem} we
compute the graded Betti numbers of $R/I$ in terms of the graded
Betti numbers of $R/J_{\mathcal{H}}$ and the graded Betti numbers of
$R/J_{\mathcal{H'}}$ for some hypergraphs $\mathcal{H'}$ associated
to $\mathcal{H}$. We believe that this approach can be more efficient than the technique of polarization in many cases. For example, when $\mathcal{H}$ is a graph, we just need to consider the edge ideals of  the graph and some induced subgraphs of it instead of working in a larger polynomial ring. To see an application of our approach, in Theorem \ref{complete graph} and  Theorem \ref{characterization complete graph} we focus to the particular case $I=J_G+(x_1^2,\ldots ,x_n^2)$ when $G=K_{n_1,\ldots,n_t}$. We compute the graded Betti numbers of $R/I$ and show that the property of being a complete $t-$partite graph for $G$  depends only to the last Betti numbers of $R/I$.

Another interesting consequence  of  Theorem \ref{hypergraph main theorem} is given in Corollary \ref{graph main theorem}. There,  we study the last (graded) Betti
numbers of $R/I$ and relate these invariants to the maximal
independent sets of  ${\mathcal{H}}$. In Corollary \ref{final
result}, for the case
$I=J_{\mathcal{H}}+(x_1^{2},\ldots, x_n^{2})$, we show that
$\beta_{n,j}(R/I)$ is equal to the number of facets of size $j-n$ in
the independent complex of ${\mathcal{H}}$. As an important consequence of it we have
 $\reg(R/I)=\alpha(\mathcal{H})$ where
$\alpha(\mathcal{H})$ is the independence number of $\mathcal{H}$.
Note that the formula of regularity, just in the case that  $\mathcal{H}$ is a graph,  also obtained by \cite[Theorem 20
and Lemma 21]{W}.

\section{\bf Preliminaries}

Throughout this paper, $\m=(x_1,\ldots,x_n)$ is the unique maximal
graded ideal of $R$ and  the set $\{1, \ldots,n\}$ is denoted by
$[n]$.

 For a graded $R-$module $M$, let $\{\beta_{i,j}(M)\}$ be the sequence of the graded Betti numbers of $M$, the {\it Castelnuovo-Mumford
 regularity } of $M$ is defined as
 $$\reg(M)=\max\{j-i\ ; \ \beta_{i,j}(M)\neq 0\},$$
 and the {\it projective dimension} of $M$ is defined as
 $$\pd(M)=\max\{i\ ; \ \beta_{i,j}(M)\neq 0 \ {\text{for some}}\
 j\}.$$
 By {\it Auslander-Buchsbaum formula} (see \cite[Theorem 19.9]{E1}, one
 has  $\pd(M)+\depth(M)=n.$

\begin{remark}\label{sfree pd}\rm
 For a  squarefree monomial ideal $I\subsetneq \m$, since $\m$ does
 not belong to the set of associated primes of $I$, we always have
 $\depth(R/I)>0$ and consequently, $\pd(R/I)<n$.
 \end{remark}

 Let $I=(x_{11}\cdots x_{1n_1},\ldots,x_{t1}\cdots x_{tn_t})$ be a
 squarefree monomial ideal, the {\it Alexander dual ideal} of $I$,
 denote $I^\vee$, is defined as
 $$I^\vee=(x_{11},\ldots, x_{1n_1})\cap\ldots\cap(x_{t1},\ldots,
 x_{tn_t}).$$

  For a graded   $R$-module $M$ and $d \in {\mathbb{Z}} $  we write $ M_{<d>} $ for
the submodule of $M$  which is generated by all homogeneous elements
of $ M $ with degree $ d$.  We say that $ M$ has a $d$-linear
resolution if $\beta_{i, j}(M) = 0 $ for $j \neq d+i$ and we say $
M$  is componentwise linear if for all integers $d $ the module $
M_{<d>} $ has a $d$-linear resolution.

 \begin{definition}\rm
A graded $R-$module $M$ is called {\it sequentially Cohen-Macaulay}
if there exists a finite filtration of graded $R-$modules
$$0=M_0\subset M_1\subset\cdots\subset M_r=M$$
such that each $M_i/M_{i-1}$ is Cohen-Macaulay and
$$\dim(M_1/M_0)<\dim(M_2/M_1)<\cdots<\dim(M_r/M_{r-1}).$$
\end{definition}
\begin{theorem}\label{dual}
Let $I$  be a squarefree monomial ideal. Then
\begin{enumerate}
    \item $\pd(R/I)=\reg(I^\vee)$ (\cite[Theorem 2.1]{T}).
     % \item $R/I$ is Cohen-Macaulay if and only if
    %$I^\vee$ has linear resolution(Eagon-Reiner Theorem, \cite{ER}, see
    %also \cite[Theorem 8.1.9]{HH2}).
    \item $R/I$ is sequentially Cohen-Macaulay if and only if
    $I^\vee$ is componentwise linear(\cite{HH}, see also \cite[Theorem
    8.2.20]{HH2}).
    \end{enumerate}
\end{theorem}

\vspace{3mm}

\subsection*{Iterated mapping cone}

In the following we recall the mapping cone technique from \cite{HT}.
Let  $\{f_1,\ldots , f_m\}$ be a homogeneous system of generators
for $I$, and $I_j = (f_1, \ldots , f_j)$. Then for $j = 2,\ldots, m$
there are exact sequences $$0\to R/(I_{j-1} : f_j) \to R/I_{j-1} \to
R/I_j\to  0.$$

\noindent Assuming that a free $R-$resolution $({\bf{F.}},\delta.)$
of $R/I_{j-1}$
 and a free $R-$resolution $({\bf{G.}},d.)$ of $R/(I_{j-1} : f_j)$ are
 known, we can obtains a resolution $({\bf M(\psi)},\gamma.)$ of $R/I_j$ as a {\it mapping cone} of a
complex homomorphism $\psi: {\bf{G.}}\to {\bf{F.}}$ which is a
lifting of the map $ R/(I_{j-1} : f_j) \to R/I_{j-1}$. The mapping
cone ${\bf M(\psi)}$ is the complex such that $$(
M(\psi))_i=F_i\oplus G_{i-1},$$ with the differential maps
$$\gamma_i(x,y)=(\psi_{i-1}(y)+\delta_{i}(x),-d_{i-1}(y))$$ where $x\in
F_i$ and $y\in G_{i-1}$. This complex is exact (see \cite[Page 650
and Proposition A3.19.]{E1}), so, it is a free resolution for
$R/I_j$.

Of course, in general,  such a resolution may be  non-minimal. But in
any case this method yields an inductive procedure to compute a
resolution of $R/I$ provided for each $j$, a resolution of
$R/(I_{j-1} : f_j)$ is known as well as the comparison map.

\vspace{3mm}

Next, we give a sufficient condition to check the
minimality of the resolution obtained by the mapping cone technique
for $R/I+(f)$ where $I$ is a graded ideal and $f$ is a homogeneous
polynomial.

 We remark  that this  result is a generalization
of \cite[Theorem 2.7]{BHK} where the authors study the minimal free
resolution of the path ideal of a rooted tree.

\begin{theorem}\label{criteria1}
Let $I$ be a graded  ideal of $R$ and $f$ is a homogeneous
polynomial  of degree $d$ which does not belong to $I$ then we have
the following graded short exact sequence
$$0\to R/(I: f)(-d) \to R/I \to R/I+(f)\to 0.$$ Assuming that the
minimal free resolution  of the modules $R/(I:f)$ and $R/I$ are
already known. Then the minimal free resolution  of $R/I+(f)$ is
obtained by the mapping cone provided that $f=h_1h_2$ where $h_1$
and $h_2$ are  homogeneous polynomials, $\deg(h_2)>0$ and
$(I:f)=(I:h_1)$

 and in this case

\begin{description}
    \item[(a)]
$$ \beta_{ij}(R/I+(f))=\beta _{ij}(R/I)+\beta_{i-1 j-d}(R/(I:f)),$$
\item[(b)]$$ \reg(R/(I+(f))=\max\{ \reg(R/I), \reg(R/(I:f))+d-1\}$$
\item[(c)]$$ \pd(R/(I+(f))=\max\{ \pd(R/I), \pd(R/(I:f))+1\}.$$
\end{description}
\end{theorem}

\begin{proof}
%Assume that $\forall\  v\in G(I)\ \deg_{x_i}(u)>\deg_{x_i}(v)$ and
%let $u'=u/x_i$. It is clear that $u'$ is a monomial of degree $d-1$
%which does not belong to $I$ and moreover $I:u=I:u'$.

 Let
$({\bf{F.}},\delta.)$ be the minimal free resolution of $R/I$,
$({\bf{G.}},d.)$ be the minimal free resolution of $R/(I : h_1)$
shifted by $\deg(h_1)$   and $\psi: {\bf{G.}}\to {\bf{F.}}$ be the
complex graded homomorphism
 which is a
lifting of the map $ R/(I : h_1)(-(\deg(h_1)) \to R/I$. Since
$I:f=I:h_1$, if we denote by $({\bf{G'.}},d'.)$ the shifted  by
$deg(h_2)$ of the graded complex $({\bf{G.}},d.)$, clearly we get
the minimal free resolution of $R/(I : f)$ shifted by $d$. Moreover
$\psi'=h_2\psi: {\bf{G'.}}\to {\bf{F.}}$ is the complex graded
homomorphism
 which is a
lifting of the map $ R/(I : u)(-d) \to R/I$.

Let for each $r$, $M_r$ (resp. $N_r$) be the matrix of $\delta_r$
(resp. $d'_r$) with respect to the canonical basis of $F_r$ and
$F_{r-1}$ (resp. $G'_r$ and $G'_{r-1}$). Also assume that for each
$r$, $O_r$ be the matrix of $\psi'_r: G'_r\to F_r$. Then, by mapping
cone construction, the matrix of $\gamma_r$, with respect to the
canonical  basis of $F_r\oplus G'_{r-1}$ and  $F_{r-1}\oplus
G'_{r-2}$, is denoted by $M'_r$ has the following shape;

\begin{displaymath}
   M'_r =\left (\begin{array}{c|c}
       \begin{array}{cccc}
     M_r                 \end{array}
 &  \begin{array}{cccc}
  O_{r-1}            \end{array}

\\
             \hline
         \begin{array}{cccc}
     0                \end{array}
 &  \begin{array}{cccc }
    -N_{r-1}
         \end{array}  \end{array}\right)
   \end{displaymath}

So, the result of the mapping cone is the minimal free resolution if
and only if   $Im(\psi')\subset \m \bf{F.}$. This clearly holds
since $\psi'=h_2\psi$, and $h_2\in \m$.
\end{proof}

\begin{example}{\rm
Let $I=(x^3y^5,xy^5z^6)\subset R={\bf{k}}[x,y,z]$ and $f=xyz^7-z^9$.
Then $f=z^7(xy-z^2)$,
$I:f=((I:z^7):(xy-z^2))=((xy^5):xy-z^2)=(xy^5)=I:z^7$. So  Theorem
\ref{criteria1} shows that we can compute the minimal free resolution of $R/I+(f)$  by
the mapping cone technique. Note that $I$ is a monomial ideal generated by $x^3y^5$ and $xy^5z^6$. It is easy to see that the set $$\{(g_1,g_2)\in R^2\ ; \ g_1x^3y^5+g_2xy^5z^6=0\}$$ is the submodule of $R^2$ generated by $(z^6,-x^2)$. So the minimal free resolution of $R/I$ is $$0\to R(-14)\to R(-8)\oplus R(-12)\to R\to R/I\to 0.$$
It is also clear that the minimal free resolution of $R/I:f(-\deg(f))=R/(xy^5)(-9)$ is $$0\to R(-15)\to R(-9)\to R/I:f(-\deg(f))\to 0.$$
So, by the mapping cone, the minimal free resolution of $R/I+(f)$ is $$0\to R(-14)\oplus R(-15)\to R(-8)\oplus R(-12)\oplus R(-9)\to R\to R/I+(f)\to 0.$$}
\end{example}

We remark that  if $u=x_1^{\alpha_1}\cdots x_n^{\alpha_n}\in R$ is a
monomial, then $\supp(u)=\{i\ ; \ \alpha_i>0\}$ and
$\deg_{x_i}(u)=\alpha_i$.
 For a
monomial ideal $I$, the unique minimal system of generators for $I$
denoted by $G(I)$.
 In the following, when we write $I=(u_1,\ldots , u_m)$ it means
that $I$ is a monomial ideal and $G(I)=\{u_1,\ldots ,u_m\}$.

\begin{corollary}\label{criteria}
If $I$ is a monomial ideal of $R$ and $u$ is a monomial which does
not belong to $I$, then the minimal free resolution of $R/I+(u)$ is
given by the mapping cone technique provided that
$$\exists \ x_i\in \supp(u)\  {\text{such that}}\  \forall\  v\in
G(I)\ \deg_{x_i}(u)>\deg_{x_i}(v).$$
\end{corollary}
\begin{proof}
By assumption, for some  $x_i\in \supp(u)$ we can decompose $u$ as
$u=u_1x_i$ such that $I:u=I:u_1$, So the result follows by Theorem
\ref{criteria1}.
\end{proof}

\begin{definition}\label{dec type def}\rm
Let $I=(u_1,\ldots ,u_m)$ be a monomial ideal. We say $I$ is {\it of
decreasing type}  with respect to the order $u_1,\ldots, u_m$ of its generators, if
for each $u_j$ there exists $x_i\in \supp(u_j)$ such that
$\deg_{x_i}(u_j)>\deg_{x_i}(u_r)$ for all $r<j$.
\end{definition}

For example if $I=(x_1x_2,x_2x_3,x_3^2,x_3x_4)\subset
{\bf{k}}[x_1,x_2,x_3,x_4]$, then $I$ is of decreasing type with
respect to $x_1x_2,x_2x_3,x_3^2,x_3x_4$.

Note that being of decreasing type depends to the ordering of the
generators and when we say $I=(u_1,\ldots,u_m)$ is of decreasing
type, it means that it is of decreasing type with respect to the order
$u_1,\ldots,u_m$.

 The following theorem is an immediate consequence
of Corollary \ref{criteria}.

\begin{corollary}\label{dec type}
Let $I=(u_1,\ldots ,u_m)$ be a monomial ideal of decreasing type.
Then the minimal free resolution of $I$ is given by iterated mapping
cone.
\end{corollary}

In the next sections we apply Corollary \ref{criteria} and Theorem \ref{dec
type} in different situations to study the minimal free resolution
of some classes of monomial ideals.

%%%%%%%%%%%%%%%%%%%%%%%%%%%%
%%%%%%%%%%%%%%%%%%%%%%%%%%%%%%
 %%%%%%%%%%%%%%%%%%%%%%%%%%%%%%%

\section{\bf Max-path ideals of rooted trees}

A tree is a graph in which there exists a unique path between every
pair of distinct vertices; a {\it rooted tree} is a tree together
with a fixed vertex called the root with the property that there
exists a unique path from the root to any given vertex. So a rooted
tree is a directed graph by assigning to each edge the direction
that goes away from the root.  Also an isolated vertex is
considered as a trivial rooted tree. If $\{i, j\}$ is an edge in a
rooted tree $T$, then we write $(i, j)$ for the directed edge
whose direction is from $i$ to $j$.  A directed path  is a sequence of distinct
vertices $ {i_1},\ldots,  {i_t}$ , in which $( {i_j} ,  {i_{j+1}})$
is the directed edge from $ {i_j}$ to $ {i_{j+1}}$ for any $j = 1,
\ldots , t - 1$.

we need the following definitions for a
rooted tree $T$.

\begin{definition}\rm
Let $T$ be a rooted tree.   A vertex $y$ is called a {\it child} of
x if $(x, y)$ is a directed edge in $T$.   A vertex $y \neq x$ is a
{\it descendant} of $x$ if there is a directed  path from $x$ to
$y$. The vertex $x$ is called a leaf of $T$ if $x$ has no child.
\end{definition}

\begin{definition}\rm
 Let $T$ be a rooted tree. An induced subtree
(or forest) of $T$ is a directed subtree (or forest) that is also an
induced subgraph of $T$. Let  $x$ be a vertex in $T$. The induced
subtree rooted at $x$ of $T$ is the induced subtree of $T$ on the
vertex set $\{x\} \cup \{y\  ;\  y\ {\text{is a descendant of}}\
x\}$.
\end{definition}

Next we define  and study the class of  max-path ideals of  rooted trees. This class of ideals  first defined  and studied in \cite{ShKhN} for an arbitrary  tree.   This class of ideals has interesting properties as we  see later.

\begin{definition}\rm
  Let $T$ be a rooted tree on the vertex set $[n]$. The {\it max-path ideal} of $T$ is
defined as
$$PI(T)=(x_{i_1}\cdots x_{i_t}\  ;\ {i_1} ,\ldots,  {i_t}{\text{ is
a maximal path in}} \ T) \subseteq R,$$
where by a {\it maximal path } we mean a directed path between the root of tree and one of its leaves.
\end{definition}

Here, we show that $PI(T)$ is a monomial ideal of
decreasing type and we study the numerical invariants of its minimal
free resolution by using Corollary \ref{dec type}.

\begin{theorem}\label{rooted tree main1}
Let $T$ be a rooted tree on the vertex set $[n]$. Then
\begin{itemize}
\item[(i)] The max-path ideal of $T$ is of decreasing type. So the
minimal free resolution of $R/PI(T)$ is obtained by the iterated
mapping cone.
\item[(ii)] $\dim(R/PI(T))=n-1$.
\item[(iii)] Let $m=$ The number of leaves of $T$. Then
\begin{itemize}
\item[(a)] $\beta_i(R/PI(T))={m\choose i}$.
\item[(b)]  $\pd(R/PI(T))=m$ and $\depth(R/PI(T))=n-m$.
\item[(c)]  $\reg(R/PI(T))=n-m$.
\end{itemize}
\item[(iv)]  $R/PI(T)$ is Cohen-Macaulay if and only if $T$ is a directed path.
\end{itemize}
\end{theorem}
\begin{proof}
(i): Let $ 1$ be the root of $T$ and $L(T)=\{ {i_1},\ldots, {i_m}\}$ be
the set of  leaves of $T$. For each $1\leq j\leq m$, let $u_j$ be
the monomial corresponding to the maximal path from $1$ to $ {i_j}$.
It is clear that $\deg_{x_{i_r}}(u_r)> \deg_{x_{i_r}}(u_j)$ for each
$j\neq r$. So,  $PI(T)$ is a monomial ideal of decreasing type and
by Corollary \ref{dec type},  the minimal free resolution of $R/PI(T)$
is obtained by the iterated mapping cone.

(ii): By definition of $PI(T)$, it is clear that $(x_1)$ is an associated prime of $PI(T)$. So $\dim (R/PI(T))=n-1$.

(iii): By induction on $m$ and using the mapping cone technique we
compute the desired formulas.

Let $m=1$ and $ {i_1}$ be the only leaf of $T$. So $T$ is just a
directed path and $PI(T)$ is a principle monomial ideal. So it is
clear that $\beta_i(R/PI(T))={1\choose i}$, $\pd(R/PI(T))=1$,
$\depth(R/PI(T))=n-1$ and $\reg(R/PI(T))=n-1$.

Now assume that the result is true for each rooted tree whose number
of leaves are less than $m$ and assume that $T$ is a rooted tree
with $m$ leaves. Let $u=x_{r_1}\cdots x_{r_k}\in G(I)$ where $
{r_1}= 1$ is the root of $T$, each $ {r_j}$ is a child of $
{r_{j-1}}$ and $r_k=i_m$ is a leaf. Then $PI(T)=PI(T')+(u)$ where
$T'$ is the rooted tree  that $G(PI(T'))=G(PI(T))\setminus \{
x_{r_1}\cdots x_{r_k}\}$. Note that $L(T')=\{  {i_1},\ldots,
 {i_{m-1}}\}$ and $V(T')\subseteq V(T)\setminus \{ {i_m}\}$. For
each $1\leq j\leq k-1$, let $C_j=\{x\in V(T)\ ; \ x\ {\text {is a
child of}}\  {r_j}\}\setminus \{ {r_{j+1}}\}$ and
$C=\cup_{j=1}^{k-1} C_j$.

It is easy to see that $PI(T'):u=\sum_{l=1}^\ell PI(T_l)$ where
$\ell=|C|$ and each $T_l$ is an induced subtree rooted at a vertex
of $C$. Moreover $\cup_{l=1}^\ell L(T_l)=L(T')$ and $\cup_{l=1}^\ell
V(T_l)=V(T)\setminus \{ {r_1},\ldots , {r_k}\}$.

Now let  $R_0={\bf{k}}[x_i\ ; \ i\in [n]\setminus \cup_{l=1}^\ell
V(T_l)]$ and for each $1\leq l\leq \ell$, $R_l={\bf{k}}[x_i\ ; \
i\in V(T_l)]$ . Then
$$R/(PI(T'):u)=R_0\bigotimes (\bigotimes_{l=1}^\ell R_l/PI(T_l)).$$

 By induction hypothesis we have:

\begin{align*}
\beta_i(R/(PI(T'):u))&=\sum_{l_1+\cdots+l_\ell=i}\beta_{l_1}(R_1/PI(T_1))\times\cdots\times\beta_{l_\ell}(R_\ell/PI(T_\ell)) \nonumber\\
&=\sum_{l_1+\cdots+l_\ell=i}{|L(T_1)|\choose l_1}\times \cdots
\times{|L(T_\ell)|\choose l_\ell} \nonumber\\
&={\sum_{l=1}^\ell |L(T_l)|\choose
i}={|L(T)|-1\choose i}\nonumber\\
&={m-1\choose i},
\end{align*}

\begin{align*}
\pd(R/(PI(T'):u))&=\sum_{l=1}^\ell
\pd(R_l/PI(T_l)) \nonumber\\
&=\sum_{l=1}^\ell |L(T_l)|\nonumber\\
&=|L(T)|-1\nonumber\\
&=m-1,
\end{align*}
and

\begin{align*}
\reg(R/(PI(T'):u))&=\sum_{l=1}^\ell\reg(R_l/PI(T_l)) \nonumber\\
&=\sum_{l=1}^\ell
(|V(T_l)|-|L(T_l)|)\nonumber\\
&=n-k-(m-1).
\end{align*}
 Also, for $R/PI(T')$ we have

 $$\beta_i(R/(PI(T')))={|L(T')|\choose i}={m-1\choose i},$$

$$\pd(R/(PI(T')))=|L(T')|=m-1,$$ and
$$\reg(R/(PI(T')))=|V(T')|-(m-1)\leq n-1-(m-1)\leq n-m.$$

Now we apply mapping cone to the short exact sequence
$$0\to R/(PI(T'):u)(-k)\to R/(PI(T'))\to R/(PI(T))\to 0.$$
By Theorem \ref{criteria1} we get

\begin{align*}
\beta_i(R/(PI(T))&=\beta_i(R/(PI(T'))+ \beta_{i-1}(R/(PI(T'):u)) \nonumber\\
&={m-1\choose i}+{m-1\choose
i-1}\nonumber\\
&={m\choose i},
\end{align*}
\begin{align*}\reg(R/(PI(T)))&=\max\{\reg(R/(PI(T'))),\reg(R/(PI(T'):u))+k-1\}\nonumber\\
&=n-m,\end{align*}
and
\begin{align*}\pd(R/(PI(T)))&=\max\{\pd(R/(PI(T'))),\pd(R/(PI(T'):u))+1\}\nonumber\\
&=m.\end{align*} So
the result follows.

(iv): By parts (ii) and (iii), $R/PI(T)$ is Cohen-Macaulay if and only if $m=1$. So the result is clear.
\end{proof}

In the following we are going to find some nice properties of $PI(T)$. We first need to recall the definition  of a  simplicial tree. Simplicial trees have the nice property that whose facet ideals are sequentially Cohen-Macaulay (see \cite[Corollary 5.6]{F}).
\begin{definition}\rm

A {\it simplicial complex} $\Delta$ on the vertex set $V (\Delta) =
\{x_1,\ldots ,x_n\}$ is a collection of subsets of $V (\Delta)$ such
that if $F\in \Delta$ and $G \subset F$, then $G\in \Delta$. An
element in $\Delta$ is called a {\it face} of $\Delta$, and $F \in
\Delta$ is said to be a facet if $ F$ is maximal with respect to the
inclusion. Let $F_1,\ldots, F_q$ be all the {\it facets} of a
simplicial complex $\Delta$, we write $\Delta=\langle F_1,\ldots
,F_q\rangle$.

The facet ideal of $\Delta$ is
$$I(\Delta)=(\prod_{x\in F}x\ ; \ F\ {\text {is a facet of}}\
\Delta).$$

Let $T$ be a rooted tree. Then $PI(T)$ can be considered as the
facet ideal of the following simplicial complex

$$\Delta_{PI(T)}=\langle \{x_{r_1},\ldots,x_{r_k}\}\ ; \
 {r_1},\ldots, {r_k} \ {\text{is a maximal path of }} \ T\rangle$$

A leaf of a simplicial complex $\Delta$ is a facet $F$ of $\Delta$
such that either $F$ is the only facet of $\Delta$, or there exists
a facet $G$ in $\Delta$, $G \neq F$, such that $F\cap F'\subseteq
F\cap G$ for every facet $F'\in\Delta$, $F'\neq F$. A simplicial
complex $\Delta$ is a called {\it simplicial tree} if $\Delta$ is
connected and every non-empty subcomplex $\Delta'$ contains a leaf.
By a subcomplex, we mean any simplicial complex of the form
$\Delta'=\langle F_{i_1},\ldots , F_{i_q}\rangle$ , where $\{F_{i_1}
, \ldots, F_{i_q} \}$ is a subset of the set of all facets of
$\Delta$.
\end{definition}

We next see that $R/PI(T)$ is sequentially Cohen-Macauly. This  is an immediate consequent of the following theorem which shows that  $\Delta_{PI(T)}$ is a simplicial tree.

\begin{theorem}\label{simplicial tree}
Let $T$ be a rooted tree. Then $\Delta_{PI(T)}$ is a simplicial tree
\end{theorem}
\begin{proof}
We show that each facet of $\Delta_{PI(T)}$ is a leaf. Let $P:
{r_1},\ldots, {r_k}$ be a maximal path of $T$ where $ {r_1}= 1$ is
the root of $T$ and each $ {r_j}$ is a child of $ {r_{j-1}}$. So $
{r_k}$ is a leaf of $T$. Let $F$ be the  facet of $\Delta_{PI(T)}$
corresponding to $P$. For each $1\leq i\leq k-1$, let $C_i=\{x\ ; \
x$ is a child of $ {r_{i}}\}\setminus \{ {r_{i+1}}\}$ and
$\ell=\max\{i\ ; \ C_i\neq \emptyset\}$. Let $G$ be the facet
corresponding to a maximal path $P': {r_1},\ldots, {r_\ell},
{r'_{\ell+1}},\ldots, {r'_{k'}}$ where $ {r'_{\ell+1}}$ is a child
of $ {r_\ell}$, $ {r'_{\ell+1}}\neq  {r_{\ell+1}}$ and each $
{r'_j}$ is a child of $ {r'_{j-1}}$. It is easy to see that $F\cap
F'\subseteq F\cap G$ for every facet $F'\in\Delta_{PI(T)}$, $F'\neq
F$.

Now let $\Delta'=\langle F_{i_1},\ldots , F_{i_q}\rangle$ be a
subcomplex of $\Delta_{PI(T)}$ and $V'=V(\Delta')$. If $T'$ is the
induced subtree of $T$ on the vertex set $V'$, then
$\Delta'=\Delta_{PI(T')}$. So by the previous paragraph, each facet
of $\Delta'$ is a leaf. So $\Delta_{PI(T)}$ is a simplicial tree.
\end{proof}

\begin{corollary}\label{rooted tree main 2}
 Let $T$ be a rooted tree. Then\begin{itemize}
    \item  $R/PI(T)$ is sequentially
 Cohen-Macauly.
    \item  $PI(T)^\vee$ is componentwise linear.
    \item
    $\reg(PI(T)^\vee)=m$ where $m$ is the number of leaves in $T$.
    \item $\pd(PI(T)^\vee)=n-m$.
 \end{itemize}
\end{corollary}
\begin{proof}
 By  Theorem \ref{simplicial tree}, $\Delta_{PI(T)}$ is a simplicial tree and
 $PI(T)=I(\Delta_{PI(T)})$. By (\cite[Corollary 5.6]{F}), $R/PI(T)$ is sequentially
 Cohen-Macauly. Other parts follows by Theorem \ref{dual} and the
 fact that $PI(T)^{\vee\vee}=PI(T)$.
\end{proof}
%%%%%%%%%%%%%%%%%%%%%%%%%%%
%%%%%%%%%%%%%%%%%%%%%%%%%
%%%%%%%%%%%%%%%%%%%%%%%%%%%%%%
%%%%%%%%%%%%%%%%%%%%%%%%%%%%%%%

%%%%%%%%%%%%%%%%%%%%%%%%%%%%%%
%%%%%%%%%%%%%%%%%%%%%%%%%%%%%%
%%%%%%%%%%%%%%%%%%%%%%%%%%%%%%%%%%%%
\section{\bf Monomial ideals containing some powers of variables}

Let $J$ be a monomial ideal and $I=J+ ( x_{i_1}^{a_{i_1}},\ldots,x_{i_m}^{a_{i_m}})$ where $a_{i_j}$ are positive integers and $G(I)=G(J)\cup \{x_{i_1}^{a_{i_1}},\ldots,x_{i_m}^{a_{i_m}}\}$.
 In this
section we are going to study the minimal free resolution of
$R/I$ using Theorem \ref{criteria1}.

First, we compute the graded Betti numbers of
$R/I$ in terms of the graded Betti numbers of $R/J$
and the graded Betti numbers of some other modules associated  to
$R/J$. This result has been proved in \cite{MPS} by applying mapping cone
to a long exact sequence. Here, using Theorem
\ref{criteria1}, we give an easier proof with more details for it. Next we focus to the case that $J$ is a square-free monomial ideal.

\begin{theorem}\label{gereralized peeva} Let
$J$ be a monomial ideal, $I=J+ ( x_{i_1}^{a_{i_1}},\ldots,x_{i_m}^{a_{i_m}})$  and $G(I)=G(J)\cup \{x_{i_1}^{a_{i_1}},\ldots,x_{i_m}^{a_{i_m}}\}$. Then
\begin{enumerate}
\item The minimal free resolution of $R/I$ is obtained by iterated
mapping cone starting from the minimal free resolution of $R/J$.
\item
$$\dim(R/I)\leq n-m,$$
\item
\begin{align}\label{betti}
\beta_{i,j}(R/I)&=\sum_{r=0}^m\sum_{|\sigma|=r}\beta_{i-r,
j-\ell_\sigma}(R/(J:\prod_{j\in \sigma}x_j^{a_j}))\nonumber\\
& {\text {where}}\
\sigma\subseteq \{i_1,\ldots,i_m\},\ \ell_\sigma=\sum_{t\in
\sigma}a_t.\end{align}
\end{enumerate}
\end{theorem}
\begin{proof}
1) For each $1\leq j\leq m$  and $u\in G(J)$, $\deg_{x_{i_j}}(u)< a_{i_j}$ , so part 1
is an immediate consequence of Theorem \ref{criteria1}.

 2)  Since $(
x_{i_1}^{a_{i_1}},\ldots,x_{i_m}^{a_{i_m}})\subset I$, it is clear
that $$\dim(R/I)\leq \dim(R/(
x_{i_1}^{a_{i_1}},\ldots,x_{i_m}^{a_{i_m}})\leq n-m.$$

3)  We compute the Betti numbers of $R/I$ by induction on $m$. Let
$m=1$. So, $I=J+(x_{i_1}^{a_{i_1}})$ and $G(I)=G(J)\cup \{x_{i_1}^{a_{i_1}}\}$. Therefore
$\deg_{x_{i_1}}(x_{i_1}^{a_{i_1}})>\deg_{x_{i_1}}(u)$ for each $u\in
G(J)$, and by Theorem \ref{criteria1}, the minimal free resolution of
$R/I$ is obtained by the mapping cone corresponding to the following
short exact sequence
$$0\to R/(J:(x_{i_1}^{a_{i_1}}))(-a_{i_1})\to R/J\to R/I\to 0.$$
So
$$\beta_{i,j}(R/I)=\beta_{i,j}(R/J)+\beta_{i-1,j-a_{i_1}}(R/J:(x_{i_1}^{a_{i_1}})).$$
which coincides to the Equation $(\ref{betti})$ for the case $m=1$.
Now assume that $m>1$ and  the result is true for all $k$ smaller than $m$. We prove it for  $m$. So assume that $I=J+(
x_{i_1}^{a_{i_1}},\ldots,x_{i_m}^{a_{i_m}})$ and $G(I)=G(J)\cup \{x_{i_1}^{a_{i_1}},\ldots,x_{i_m}^{a_{i_m}}\}$.

 Let $J'=J+(
x_{i_1}^{a_{i_1}},\ldots,x_{i_{m-1}}^{a_{i_{m-1}}})$. It is clear that $$G(J')=G(J)\cup \{x_{i_1}^{a_{i_1}},\ldots,x_{i_{m-1}}^{a_{i_{m-1}}}\},$$$I=J'+(x_{i_m}^{a_{i_m}})$ and $G(I)=G(J')\cup \{x_{i_m}^{a_{i_m}}\}$. Therefore,
\begin{equation}\label{betti2}\beta_{i,j}(R/I)=\beta_{i,j}(R/J')+\beta_{i-1,j-a_{i_m}}(R/J':(x_{i_m}^{a_{i_m}})).\end{equation}
 Moreover,
$J':(x_{i_m}^{a_{i_m}})=(J:(x_{i_m}^{a_{i_m}}))+(x_{i_1}^{a_{i_1}},\ldots,x_{i_{m-1}}^{a_{i_{m-1}}})$.  It is easy to see that $G(J':(x_{i_m}^{a_{i_m}}))=G(J:(x_{i_m}^{a_{i_m}}))\cup\{x_{i_1}^{a_{i_1}},\ldots,x_{i_{m-1}}^{a_{i_{m-1}}}\}$.
So by induction hypothesis for the case $m-1$,

\begin{align}\label{betti3}\beta_{i,j}(R/J')&=\sum_{r=0}^{m-1}\sum_{|\sigma|=r}\beta_{i-r,
j-\ell_\sigma}(R/(J:\prod_{j\in \sigma}x_j^{a_j}))\nonumber\\
&  {\text {where}}\
\sigma\subseteq \{i_1,\ldots,i_{m-1}\},\ \ell_\sigma=\sum_{t\in
\sigma}a_t,\end{align}

and

\begin{align}\label{betti4}&
\beta_{i-1,j-a_{i_m}}(R/(J':(x_{i_m}^{a_{i_m}})))\nonumber\\
&=\sum_{r=0}^{m-1}\sum_{|\sigma|=r}\beta_{i-1-r,
j-a_{i_m}-\ell_\sigma}(R/((J:(x_{i_m}^{a_{i_m}})):\prod_{j\in
\sigma}x_j^{a_j}))\nonumber\\
& {\text {where}}\ \sigma\subseteq
\{i_1,\ldots,i_{m-1}\},\ \ell_\sigma=\sum_{t\in
\sigma}a_t.\end{align}

For each $\sigma \subseteq \{i_1,\ldots,i_{m-1}\}$, we let
$\sigma'=\sigma\cup \{i_m\}$. It is clear that
$a_{i_m}+\ell_\sigma=\ell_{\sigma'}$, and
$(J:(x_{i_m}^{a_{i_m}})):\prod_{j\in \sigma}x_j^{a_j}=J:\prod_{j\in
\sigma'}x_j^{a_j}$. So the Equation (\ref{betti4}) can be written
as:

\begin{equation}\label{betti5}
\end{equation}
\begin{eqnarray*}
\beta_{i-1,j-a_{i_m}}(R/(J':(x_{i_m}^{a_{i_m}})))&&=\sum_{r=1}^{m}\sum_{|\sigma|=r}\beta_{i-r,
j-\ell_\sigma}(R/(J:\prod_{j\in \sigma}x_j^{a_j}))\\&& {\text
{where}}\ \{i_m\}\subseteq \sigma\subseteq \{i_1,\ldots,i_{m}\},\
\ell_\sigma=\sum_{t\in \sigma}a_t.\end{eqnarray*}

Now  it is enough to replace (\ref{betti3}) and (\ref{betti5}) in
(\ref{betti2}) to get Equation (\ref{betti}).
\end{proof}

 In the following we are going
to apply Theorem \ref{gereralized peeva} to the case that $J$ is a
square-free monomial ideal. We remark that an arbitrary square-free
monomial ideal can be considered as edge ideal of a hypergraph.

Let $X$ be a finite set and $\mathcal{E} = \{E_1,\ldots,E_s\}$ a
finite collection of non empty subsets of X. The pair $\mathcal{H} =
(X, E)$ is called a hypergraph on $X$. The elements of $X$ and
$\mathcal{E}$, respectively, are called the vertices and the edges
of the hypergraph. A hypergraph is called simple if $|E_i|\geq 2$
for all $i =1,\ldots s$ and $ E_j\subset  E_i$ only if $i = j$ . In the following we assume that $\mathcal{H}$ is a simple hypergraph.

Let $\mathcal{H}$ be a  hypergraph on the vertex set $X$. We
recall that $ W\subseteq
 X$ is  an independent set if $W$ does not contain any edge of
$\mathcal{H}$. The size of an independent set is the number of
vertices it contains.

A maximal independent set is either an independent set such that
adding any other vertex to the set forces the set to contain an edge
or the set of all vertices of the empty hypergraph. In the following we denote by $\max(\mathcal{H})$ the set of all maximal independent subsets of $\mathcal{H}$.

A maximum independent set is an independent set of largest possible
size for a given hypergraph $\mathcal{H}$. This size is called the
independence number of $\mathcal{H}$, and denoted
$\alpha(\mathcal{H})$.

For a hypergraph $\mathcal{H}$ on the vertex set $X$ , the
independence complex of $\mathcal{H}$  is defined  as:
$$\Delta(\mathcal{H}) = \{W\subset X\ | \ W \ \ {\text {is an
independent set}}\}.$$

For a hypergraph $\mathcal{H}$ with vertex set $[n]$ the edge ideal
of $\mathcal{H}$ in the polynomial ring $R$ is defined as: $$J_{
\mathcal{H}} = (\prod_{x\in E}x;\ E \ {\text {is an edge of}} \
\mathcal{H}).$$

Note that the edge ideal of a hypergraph is defined in the same way
as the edge ideal of a graph. We also remark that we can
consider$J_{ \mathcal{H}}$ as Stanley-Reisner ideal of
$\Delta(\mathcal{H})$.

\begin{remark}\label{remark graph}\rm
Let $\mathcal{H}$ be a hypergraph on the vertex set $[n]$.  Assume
that $J=J_\mathcal{H}$. For each $\sigma\subseteq [n]$ let
$N(\sigma)=\{i\in [n]\setminus \sigma\ ; \sigma\cup \{i\}\ {\text{is
not independent}}\}$,
 and $\mathcal{H}_\sigma$ be the simple
hypergraph on the vertex set $[n]\setminus(\sigma\cup N(\sigma))$
with $\mathcal{E}(\mathcal{H}_\sigma)=\{E\setminus \sigma\ ; \ E\in
E(\mathcal{H}), E\setminus \sigma \subseteq
V(\mathcal{H}_\sigma)\}$.

Assume that for each $j\in \sigma$, $a_j>0$. If $\sigma$ is not an independent set, then it is clear that
$J:\prod_{j\in \sigma}x_j^{a_j}=R$. If $\sigma$ is an independent
set, then $J:\prod_{j\in \sigma}x_j^{a_j}=(x_i\ ; \ i\in
N(\sigma))+J_{\mathcal{H}_\sigma}$. In particular, if $\sigma$ is a
maximal independent set, then $J:\prod_{j\in \sigma}x_j^{a_j}=(x_i\
; \ i\in N(\sigma))=(x_i\ ; \ i\in [n]\setminus \sigma)$.
\end{remark}

If $I=J_{\mathcal{H}}+(x_{i_1}^{a_{i_1}},\ldots, x_{i_m}^{a_{i_m}})$, then by  Theorem \ref{gereralized peeva} and Remark \ref{remark graph}, we
can write the graded Betti numbers of $R/I$ in terms of the graded Betti numbers of $R/J_{\mathcal{H}}$ and  $R/J_{\mathcal{H'}}$ for some hypergraphs associated to $\mathcal{H}$. In the following we discuss the case that $\forall j, a_{i_j}=2$.

\begin{theorem}\label{hypergraph main theorem}
Let  $\mathcal{H}$ be a hypergraph on the vertex set $[n]$. Assume
that $I=J_\mathcal{H}+(x_{i_1}^{2},\ldots ,x_{i_m}^{2})$. Then
\begin{align}
\beta_{i,j}&(R/I)=\sum_{r=0}^m\sum_{|\sigma|=r}\beta_{i-r,
j-2r}(R/(J_\mathcal{H}:\prod_{j\in \sigma}x_j^{2}))\nonumber\\
& {\text
{where}}\ \sigma\subseteq \{i_1,\ldots,i_m\}, \ \sigma\in\Delta(\mathcal{H})\ {\text{and}}\nonumber\\&R/J_\mathcal{H}:\prod_{j\in \sigma}x_j^{2}={\bf{k}}[x_i\ ; \ i\in
\sigma\cup N(\sigma)]/(x_i\ ; \ i\in N(\sigma))\otimes
{\bf{k}}[V({\mathcal{H}_\sigma})]/J_{\mathcal{H}_\sigma}.\label{betti6}\end{align}
\end{theorem}
\begin{proof}
First note that $G(I)=G(J)\cup \{x_{i_1}^2,\ldots,x_{i_m}^2\}$ So we can apply Theorem
\ref{gereralized peeva}.

 By Remark \ref{remark graph}, in order to compute the
Betti numbers of $R/I$, it is enough to consider all
$\sigma\subseteq \{i_1,\ldots,i_m\}$ where $\sigma\in\Delta(\mathcal{H})$. Also, if $\sigma$ is  an
independent set,
$$R/J_\mathcal{H}:\prod_{j\in \sigma}x_j^{2}={\bf{k}}[x_i\ ; \ i\in
\sigma\cup N(\sigma)]/(x_i\ ; \ i\in N(\sigma))\otimes
{\bf{k}}[V({\mathcal{H}_\sigma})]/J_{\mathcal{H}_\sigma}.$$
\end{proof}

\begin{corollary}\label{graph main theorem}
Let  $\mathcal{H}$ be a hypergraph on the vertex set $[n]$. Assume
that $I=J_\mathcal{H}+(x_{i_1}^{2},\ldots ,x_{i_m}^{2})$. Then

 $$\beta_{n,j}(R/I)=|\{\sigma\ ; \ \sigma\in \max(\mathcal{H}), |\sigma|=j-n,\ {\text{and}}\ \sigma\subseteq \{i_1,\ldots ,i_m\}\}|.$$

    Therefore,
    \begin{itemize}
    \item $\beta_{n}(R/I)=|\{\sigma\ ; \ \sigma\in \max(\mathcal{H})\ {\text{and}}\   \sigma\subseteq \{i_1,\ldots
    ,i_m\}\}|.$
    \item $\depth(R/I)=0$ if and only if $\{i_1,\ldots,i_m\}$ is
    containing a maximal independent set.
    \end{itemize}

 %   In particular if $a_1=\cdots=a_n=2$, then
  %  $\reg(R/I)=\alpha(G)$ and for each $j$, $\beta_{n j}(R/I)=$the number of
   % maximal independent sets of size $j-n$ for $G$.
\end{corollary}
\begin{proof}

$\beta_{n,j}(R/I)$ can be computed by Equation (\ref{betti6}). If
$\sigma\subseteq\{i_1,\ldots,i_m\}$ is an independent, we have
 \begin{align*}&\pd(R/J_\mathcal{H}:\prod_{j\in \sigma}x_j^{2})=\nonumber\\
&\pd({\bf{k}}[x_i\ ; \ i\in
\sigma\cup N(\sigma)]/(x_i\ ; \ i\in N(\sigma))+\pd(
{\bf{k}}[V({\mathcal{H}_\sigma})]/J_{\mathcal{H}_\sigma}),\end{align*} (where
in the above formula $
{\bf{k}}[V({\mathcal{H}_\sigma})]/J_{\mathcal{H}_\sigma}$ appears
when $V(\mathcal{H}_\sigma)\neq \emptyset$ and in this case, by
Remark \ref{sfree pd}, $\pd(
{\bf{k}}[V({\mathcal{H}_\sigma})]/J_{\mathcal{H}_\sigma}\leq
|V(\mathcal{H}_\sigma)|-1$). Therefore, if $\sigma$ is a maximal
independent set then
$$\pd(R/J_\mathcal{H}:\prod_{j\in \sigma}x_j^{2})=|N(\sigma)|=n-|\sigma|,$$
and if $\sigma$ is not a maximal independent set
\begin{eqnarray*}
\pd(R/J_\mathcal{H}:\prod_{j\in \sigma}x_j^{2})&&=|N(\sigma)|+\pd(
{\bf{k}}[V({\mathcal{H}_\sigma})]/J_{\mathcal{H}_\sigma})\leq|N(\sigma)|+|V(\mathcal{H}_\sigma)|-1\\&&
\leq
|N(\sigma)|+n-(|\sigma|+|N(\sigma)|)-1=n-|\sigma|-1.\end{eqnarray*}
So

\begin{align*}
\beta_{n,j}(R/I)&=\sum_{r=0}^m\sum_{|\sigma|=r}\beta_{n-r,
j-2r}(R/(J_\mathcal{H}:\prod_{j\in \sigma}x_j^{2}))\nonumber\\
& ({\text
{where}}\ \sigma\subseteq \{i_1,\ldots,i_m\}, \sigma\in \max(\mathcal{H}))\nonumber\\
&=\sum_{\sigma\subseteq \{i_1,\ldots,i_m\}, \sigma\in \max(\mathcal{H})}
\beta_{n-|\sigma|, j-2|\sigma|}(R/(x_i\ ; \ x_i\in N(\sigma)) )
\nonumber\\
&=|\{\sigma\ ;
\sigma\subseteq \{i_1,\ldots,i_m\}, \sigma\in \max(\mathcal{H}) {\text{ and }}\
2|\sigma|+|N(\sigma)|=j\}|\nonumber\\
&=|\{\sigma\ ; \ \sigma\subseteq \{i_1,\ldots,i_m\}, \sigma\in \max(\mathcal{H}), |\sigma|=j-n\}|.\end{align*}
 Therefore
$$\beta_n(R/I)=\sum_{j}\beta_{n,j}(R/I)=|\{\sigma\ ; \ \sigma\in\max(\mathcal{H}),   \sigma\subseteq \{i_1,\ldots
    ,i_m\}\}|.$$
\end{proof}

    \begin{corollary}\label{final result}
Let  $\mathcal{H}$ be a hypergraph on the vertex set $[n]$. Assume
that $I=J_\mathcal{H}+(x_1^{2},\ldots, x_n^{2})$. Then
\begin{enumerate}
\item
$$\beta_{n,j}(R/I)=|\{\sigma\ ; \ \sigma\in\max(\mathcal{H}), |\sigma|=j-n\}|$$
\item \begin{align*}\beta_n(R/I)&={\text{the number of  maximal independent sets for}}\ \mathcal{H} \nonumber\\
& ={\text{the number of facets of }}\ \Delta(\mathcal{H}).\end{align*}
\item $R/I$ is a level ring if and only if $\Delta(\mathcal{H})$ is
a pure simplicial complex if and only if $J_\mathcal{H}$ is an
unmixed ideal.
\item $$\reg(R/I)=\alpha(\mathcal{H}).$$
\end{enumerate}
\end{corollary}
\begin{proof}
(1) and (2)  are immediate consequences  of Corollary \ref{graph
main theorem}.

To see (3) note that $R/I$ is a level ring if and only if the last
nonzero graded  free module  of its graded minimal free resolution,  is of the form $R^a(-s)$, for some
positive integers $a$ and $s$. So by part (1), $R/I$ is a level ring
if and only if all maximal independent sets of $\mathcal{H}$ are of
the same size. Also note that $J_{\mathcal{H}}$ is unmixed if all
 minimal vertex covers of $\mathcal{H}$ have the
same cardinality. So the conclusion follows from the fact that
$C\subset [n]$ is a minimal vertex cover if and only if
$[n]\setminus C$ is a maximal independent set.

To prove (4) it is enough to notice that $\dim(R/I)=0$ and therefore
$$\reg(R/I)=\max\{j\ ; \ \beta_{n,n+j}(R/I)\neq
0\}=\alpha(\mathcal{H}).$$
\end{proof}

\begin{remark}\label{second remark for graphs}\rm
 Let $I$ be a monomial ideal generated in degree $2$ and
$I_*$ be the square-free part of $I$. It is clear that there exists
a graph $G$ on the vertex set $[n]$ in such a way that $I_*=J_G$. So
$I=J_G+(x_{i_1}^2,\ldots,x_{i_m}^2)$ for some
$\{i_1,\ldots,i_m\}\subseteq [n]$. Theorem \ref{hypergraph main
theorem} shows that we can compute the graded Betti numbers of $R/I$
in terms of the graded Betti numbers of $R/J_G$ and the graded Betti
numbers of $R/J_H$ (For some induced subgraphs $H$ of $G$).

It is also possible to study the Betti numbers of $R/I$ by the idea
of {\it polarization} (see \cite[Corollary 1.6.3]{HH2}). Note that
if $I$ and $G$ be as above, and
$$J=I_*+(x_{i_1}y_1,\ldots,x_{i_m}y_m)\subset
{\bf{k}}[x_1,\ldots,x_n,y_1,\ldots,y_m]$$ be its polarization, then
we can view $J$ as the edge ideal of the graph $H$ that is defined
as $V(H)=[n]\cup \{-1,\ldots ,-m\}$ and $E(H)=E(G)\cup
\{\{i_1,-1\},\ldots ,\{i_m,-m\}\}$. It means that $J=J_H$. Here, $G$
is an induced subgraph of $H$. The idea of attaching the graph $H$
to the ideal $I$ in order to study the Betti numbers has been used
in \cite{HHZh} where the authors studied the class of monomial
ideals with $2-$linear resolution (see \cite[Section 2]{HHZh}).

Note that by \cite[Corollary 1.6.3]{HH2},
$$\forall i,j, \beta_{ij}(R/I)=\beta_{ij}({\bf{k}}[x_1,\ldots,x_n,y_1,\ldots,y_m]/J).
$$

By \cite[Theorem 20 and Lemma 21]{W}
$\reg({\bf{k}}[x_1,\ldots,x_n,y_1,\ldots,y_m]/J)=\alpha(G)$. Since
$\reg(R/I)=\reg({\bf{k}}[x_1,\ldots,x_n,y_1,\ldots,y_m]/J)$, part
(4) of Corollary \ref{final result} is a generalization of the
mentioned result of \cite{W} to the case that
$I=J_{\mathcal{H}}+(x_1^2,\ldots,x_n^2)$ and $\mathcal{H}$ is a
hypergraph on the vertex set $[n]$.
\end{remark}

Finally, we are going to apply Theorem \ref{hypergraph main theorem} and Corollary \ref{final result} to the case that $G$ is a complete $r-$partite graph.

\begin{theorem}\label{complete graph}
Let $G=K_{n_1,\ldots,n_t}$ be a complete $t-$partite graph on the vertex set $[n]$ and let $I=J_G+(x_1^2,\ldots ,x_n^2)$. Then
$$\beta_{i,j}(R/I)=\beta_{i,j}(R/J_G)+\sum_{\ell=1}^{t}{{n_{\ell}}\choose{j-i}}{{n-n_\ell}\choose {2i-j}}$$
where \begin{align*}&\beta_{i,j}(R/J_G)= \nonumber\\
&
\left\{
  \begin{array}{ll}
   \sum_{\ell=2}^{t}(\ell-1)\sum_{\alpha_1+\cdots+\alpha_\ell=i+1, \ j_1<\cdots <j_\ell,\alpha_1,\ldots, \alpha_\ell\geq 1}{n_{j_1}\choose \alpha_1}\cdots{n_{j_\ell}\choose \alpha_\ell}, & \hbox{if $j=i+1$} \\
   0, & \hbox{if $j\neq i+1$}.
  \end{array}
\right.
\end{align*}
\end{theorem}
\begin{proof}
Assume that $G$ is $t-$partite graph with partitions $V_1,\ldots, V_t$ where $|V_\ell|=n_\ell$. Then $\sigma\subseteq [n]$ is an independent set if and only if  $\sigma\subseteq V_\ell$ for some $1\leq \ell\leq t$. So if $\sigma\neq \emptyset$ is andependent set, then  for some $1\leq \ell \leq t$, $N(\sigma)=[n]\setminus V_\ell$ and $G_\sigma$ is the empty graph on the vertex set $V_\ell\setminus \sigma$.

Now if $\sigma\subseteq V_\ell$ and $|\sigma|=r$, by Theorem \ref{hypergraph main theorem} we have
$$\beta_{i-r,j-2r}(R/J_G:\prod_{j\in \sigma}x_j^2)=\beta_{i-r,j-2r}({\bf{k}}[x_i\ ; \ i\in [n]\setminus(V_\ell\setminus \sigma)]/(x_i\ ; \ i\in [n]\setminus V_\ell)).$$ Thus
$\beta_{i-r,j-2r}(R/J_G:\prod_{j\in \sigma}x_j^2)\neq 0$ if and only if $j-2r=i-r=2i-j$ and if this is the case, then $r=j-i$ and $\beta_{i-r,j-2r}(R/J_G:\prod_{j\in \sigma}x_j^2)= {n-n_\ell\choose 2i-j}$. Now the result follows from Thorem \ref{hypergraph main theorem} and  \cite[Theorem 5.3.8]{J}.
\end{proof}

\begin{theorem}\label{characterization complete graph}
Let $G$ be a graph on the vertex set $[n]$, $I=J_G+(x_1^2,\ldots ,x_n^2)$ and $\beta_n(R/I)=t$. Then $$\sum_{j\in \mathbb{N}} \beta_{n,j}(R/I)(j-n)=n \Leftrightarrow G \ {\text {is complete}}\ t-{\text{partite graph}}.$$
\end{theorem}
\begin{proof}
If $G$ is a complete $t-$partite graph with partitions $V_1,\ldots ,V_t$, then $V_1,\ldots ,V_t$ are the only maximal independent sets of $G$. So by Corollary \ref{final result}, we have $\sum_{j\in \mathbb{N}} \beta_{n,j}(R/I)(j-n)=n$.

Conversely, Let $G$ be a graph on the vertex set $[n]$ with maximal independent sets $V_1,\ldots V_t$. Assume that $\sum_{j\in \mathbb{N}} \beta_{n,j}(R/I)(j-n)=n$. Since each vertex of the graph belongs to at least one independent  set, this equality  beside Corollary \ref{final result} show that each vertex belongs to exactly one of the independent sets. So $G$ is a complete $t-$ partite graph whose partitions are $V_1,\ldots ,V_t$.
\end{proof}

\section*{\bf Acknowledgments}
 This research was
in part supported by a grant from IPM (No. 94130058). The author would like to thank Rashid Zaare-Nahandi and Somayeh Moradi for reading an earlier version of the paper and for helpful comments and remarks.

%------------------------------------------------------------------------------------%

\end{document}